\documentclass[a4paper,11pt,reqno]{amsart}
\usepackage{amsaddr}
\usepackage[text={400pt,660pt},centering]{geometry}

\usepackage{etoolbox}
\makeatletter
\patchcmd{\@setauthors}{\MakeUppercase}{}{}{}

\makeatother
\usepackage{amsthm, amssymb, amsmath, amsfonts, mathrsfs}
\usepackage{mathtools}
\usepackage{scalerel} 
\usepackage[colorlinks=true, pdfstartview=Fit, linkcolor=blue, citecolor=blue, urlcolor=blue,pagebackref=false]{hyperref}
\hypersetup{pdftitle={Harnack inequality for p-harmonic functions: probabilistic and analytic approaches},pdfauthor={Yuval Peres and Han Wang}}

\usepackage{esint} % average integral notation

\usepackage{enumerate}

\usepackage{graphicx}

\usepackage{microtype}

\usepackage{enumitem}
\setlist[enumerate,1]{label=(\alph*), font=\normalfont\bfseries}

\usepackage{tikz} % for tikzpicture
\usetikzlibrary{shapes,arrows,positioning,calc}

\tikzstyle{result} = [rectangle, 
minimum width=4cm, 
minimum height=1cm,  
text width=4cm, 
align=flush center,
font = \footnotesize, 
draw=black]

\tikzstyle{input} = [rectangle, 
minimum width=3cm, 
minimum height=1cm, 
text width=3cm, 
align=flush center,
font = \footnotesize,
draw=black]

\tikzstyle{arrow} = [thick, ->,>=stealth]

 \makeatletter
\newcommand*\bigcdot{\mathpalette\bigcdot@{.5}}
\newcommand*\bigcdot@[2]{\mathbin{\vcenter{\hbox{\scalebox{#2}{$\m@th#1\bullet$}}}}}
\makeatother

\newtheorem{proposition}{Proposition}
\newtheorem{theorem}[proposition]{Theorem}
\newtheorem{lemma}[proposition]{Lemma}
\newtheorem{corollary}[proposition]{Corollary}

\theoremstyle{remark}
\newtheorem{remark}[proposition]{Remark}

\theoremstyle{definition}

\numberwithin{equation}{section}
\numberwithin{proposition}{section}

\renewcommand{\le}{\leqslant}
\renewcommand{\ge}{\geqslant}
\renewcommand{\leq}{\leqslant}
\renewcommand{\geq}{\geqslant}

\renewcommand{\subset}{\subseteq}

\newcommand{\F}{\mathcal{F}}

\newcommand{\E}{\mathbb{E}}

\renewcommand{\L}{\mathcal{L}}

\newcommand{\R}{\mathbb{R}}

\renewcommand{\P}{\mathbb{P}}
\newcommand{\ov}{\overline}

\renewcommand{\tilde}{\widetilde}

\newcommand{\ep}{\varepsilon}
\renewcommand{\d}{{\mathrm{d}}}

\renewcommand{\epsilon}{\varepsilon}

\renewcommand{\div}{\mathop{\rm div}}
\DeclareMathOperator*{\argmax}{argmax}
\DeclareMathOperator*{\osc}{osc}

\newcommand{\Rd}{{\mathbb{R}^d}}

\DeclareMathOperator{\dist}{dist}

\newcommand{\bracket}[1]{\left\langle{#1}\right\rangle}
\numberwithin{equation}{section}

\allowdisplaybreaks

\newcommand{\vertiii}[1]{{\left\vert\kern-0.25ex\left\vert\kern-0.25ex\left\vert #1 
	\right\vert\kern-0.25ex\right\vert\kern-0.25ex\right\vert}}

\title[]{ H\MakeLowercase{arnack inequality for}  $p$\MakeLowercase{-harmonic functions: \\ 
 probabilistic and analytic approaches}}
\author{Y\MakeLowercase{uval} P\MakeLowercase{eres\textsuperscript {*} and }{H\MakeLowercase{an} W\MakeLowercase{ang\textsuperscript{$\diamond$}}}}

\address{\textsuperscript {*}Beijing Institute of Mathematical Sciences and Applications; \\\textsuperscript{$\diamond$}Qiuzhen College, Tsinghua University}
\email{yperes@bimsa.cn}
\email{wanghan21@mails.tsinghua.edu.cn}

\begin{document}

%\begin{nouppercase}
%\end{nouppercase}

\begin{abstract}
We survey analytic and probabilistic approaches to the Harnack inequality for $p$-harmonic functions, $p>1$, with particular attention to the dependence of the constants on $p$ and the dimension. We explain the ideas of Moser iteration, the tug-of-war approach of Luiro, Parviainen and Saksman, and the log-gradient estimate of Kotschwar and Ni. Along the way, we give a quantitative refinement of the probabilistic argument and a direct planar proof. 

\bigskip


 %We prove the H\"older continuity of $p$-harmonic functions in bounded domains in $\R^d$ for $p>1$ via the tug of war with noise game. This implies the Harnack inequality of $p$-harmonic functions. We give the exact dependence on $p$ and the dimension $d$ of the H\"older constant and the Harnack constant. In particular, we show the Harnack constant is $O(\exp(C_pd\log d))$ as $d\rightarrow\infty$, which is an improvement of the constant from Moser iteration. 
\end{abstract}
\maketitle
\section{Introduction}
The purpose of this article is to explain and compare several approaches to a classical regularity theorem. The intended reader has a background in analysis or probability but need not be familiar with both the nonlinear PDE and stochastic-game viewpoints. We include the quantitative refinements and technical details needed to make the probabilistic arguments complete, while keeping the comparison of methods as the main theme.
 Harnack's inequality~\cite{harnack1887grundlagen}, extended by Serrin~\cite{Serrin1955}, is a key tool in the regularity theory of differential equations.   
%relating the values of a positive harmonic function at two points. The inequality is proved for solutions of general second order elliptic equations in \cite{Serrin1955}. 
Moser~\cite{Moser1961,Moser1964} developed an iteration method to  prove  the Harnack inequality.  
%The Moser iteration method has become a fundamental method to derive the Harnack inequality for general equations. In particular, 
Trudinger~\cite{Trudinger1967} subsequently used  this method to prove Harnack's inequality for a class of quasilinear elliptic equations that includes the $p$-Laplace equation
\begin{equation}\label{plaplace}
    \Delta_p u:=\div(|\nabla u|^{p-2}\nabla u)=0.
\end{equation}
Viscosity solutions of \eqref{plaplace} in a domain, which coincide with weak solutions by \cite{equivalenceviscocityweak}, 
 are called  $p$-harmonic functions; see Lindqvist~\cite{Lindqvistbook} or Heinonen et al.~\cite{HKMbooknonlinearpotentialtheory} for background. We now state the Harnack inequality for $p$-harmonic functions. The theorem appears in the references cited above, but the dependence of the constant on $p$ and the dimension $d$ is not made explicit there. 

\begin{theorem}[Harnack inequality]
    Let $d\geq 2$, $p>1$ and $R>1$. There exists a constant $H_R(p,d)$ such that
     every $p$-harmonic function  $u>0$ in the $d$-dimensional  ball $B_R(0)\subset\Rd$ of radius $R$ satisfies the Harnack inequality:    \begin{equation}\label{eq.harnack}
            \sup_{B_1(0)}u\leq H_R(p,d)  u(0).
        \end{equation}
\end{theorem}
Let  $H_R(p,d)$  denote the optimal constant in the above inequality. We will   compare several  analytic and probabilistic approaches   to this result and present an improvement of the probabilistic estimate. 
%In particular, we discuss the bound of the Harnack constant $H_R(p,d)$ obtained from the methods. 
We are particularly interested in the asymptotic behavior of $H_R(p,d)$ when $p\downarrow 1$ and $d\rightarrow\infty$.  Table \ref{table.comparison} compares the bounds furnished by the different methods in the regime $p\in(1,p_0]$ for a fixed constant $p_0$. For simplicity, we may just take $p_0=2$ to avoid the dependence on $p$ when $p$ is large.

\begin{table}[h!]
\centering
\renewcommand{\arraystretch}{1.4}

\begin{tabular}{|c|p{5cm}|c|}
\hline
\textbf{Method} & \textbf{Main idea} & \textbf{Harnack  upper bound} \\
\hline

Moser iteration &
Energy estimates + iteration of $L^p$ bounds &
$\exp\left(\frac{e^{Cd}}{p-1}\right)$ \\

\hline

Tug-of-war with noise &
Stochastic game representation of $p$-harmonicity &
$\exp\left(\frac{Cd}{p-1}\log\frac{Cd}{p-1}\right)$ \\

\hline

Log-gradient estimate&
Gradient estimate for $\log u$ via differential inequalities &
$\exp\left(\frac{Cd}{p-1}\right)$ \\

\hline
\end{tabular}
\smallskip
\caption{Comparison of main approaches to the Harnack inequality for $p$-harmonic functions when $p\in(1,p_0]$.}
\label{table.comparison}
\end{table}

The three approaches have rather different strengths. Moser iteration is the most classical and broadly applicable method: because it is based on energy estimates, Sobolev inequalities, and an iteration scheme, it extends well beyond the $p$-Laplace equation to large classes of linear and quasilinear elliptic and parabolic equations. The tug-of-war approach is more closely tied to the existence of a suitable stochastic-game representation, but in the present setting it yields substantially better dependence of the Harnack constant on $d$ than Moser iteration. Its probabilistic formulation is also flexible in situations where a natural dynamic programming principle is available, including variable-exponent equations such as the normalized  $p(x)$-Laplace equation and discrete settings such as graphs. The log-gradient method gives the best dependence on $p$ and $d$ among the three bounds displayed in Table~\ref{table.comparison}, and this dependence is optimal in the regime considered below. On the other hand, it relies essentially on differentiating the equation, establishing a Bochner-type formula, and exploiting the underlying smooth structure.

\subsection{Moser iteration}
In 1967, Trudinger \cite{Trudinger1967} proved the Harnack inequality for a large class of elliptic operators using Moser iteration. (For a simplification in the $p$-harmonic case, see \cite{Lindqvistbook}.)

The Moser iteration approach yields $H_R(p,d)\leq \exp\left(\frac{C_Rp}{p-1}\exp(Cd)\right)$ in \eqref{eq.harnack}. Here $C_R$ is a constant depending only on $R$. Here we restrict attention to the regime $p\in(1,2]$.

To see this, one may follow the proof presented in \cite[Section 3.3]{Lindqvistbook}. Here we explain only the point where the double exponential constant appears.  Lemma 2.14 in that book yields
\begin{equation}\label{eq.logCacci}
    \int_{B_r}|\nabla \log v|^p\leq C\left(\frac{p}{p-1}\right)^p r^{-p}(2r)^{d}\omega_d
\end{equation}
for every positive $p$-superharmonic function $v$ in $B_{2r}$, where $\omega_d$ is the volume of the $d$-dimensional unit ball, and $C$ is a universal constant.  Theorem 3.13 in \cite[Section 3.3]{Lindqvistbook} yields
\begin{equation*}
    \left(\fint_{B_r} u^\epsilon\right)^{1/\epsilon}\leq4^{1/\epsilon}\left(\fint_{B_r} u^{-\epsilon}\right)^{-1/\epsilon}
\end{equation*}
with $\epsilon=\nu(d) C_1(d,p)^{-1/p} C_2(d,p)^{-1/p}$, where $\nu(d)$ is the constant appearing in the John-Nirenberg inequality \cite[Theorem 3.13]{Lindqvistbook}, $C_1(d,p)$ is the $L^p$ Poincar\'e constant in the unit ball, i.e., $\Vert f-(f)_{B_1}\Vert _{L^p(B_1)}^p\leq C_1\Vert \nabla f\Vert _{L^p(B_1)}^p$, and $C_2(d,p)=C\left(\frac{p}{p-1}\right)^p2^d$ is the constant appearing in \eqref{eq.logCacci}. Integrating the tail estimate from \cite[Theorem 19.5]{HKMbooknonlinearpotentialtheory} shows that one may choose $\nu(d)=Ce^{-cd}$ for universal constants $C,c$. The Poincar\'e constant $C_1(d,p)$ is universally bounded in $d$ and $p\in[1,2]$. Therefore, $\epsilon^{-1}\leq C(p-1)^{-1}e^{cd}$ for a universal $C$. This gives the double exponential constant in the Moser iteration method. 

\subsection{Tug-of-war approach}
In 2008, Peres and Sheffield~\cite{PS08}
developed an approach to the $p$-Laplace equation via stochastic games called tug-of-war with noise. 
In 2013, Luiro, Parviainen and Saksman~\cite{LPSharnack} used this approach  to give a new proof of H\"older continuity and the Harnack inequality for $p$-harmonic functions when $p>2$. In \cite{LPregular} this was   extended
to a larger class of functions, 
%show H\"older continuity of functions satisfying a dynamical programming principle, 
which includes $p$-harmonic functions for $p>1$. In \cite{Lipschitzvaryingprobability}, the method was used to prove asymptotic local Lipschitz estimates for tug-of-war games with varying probabilities. In \cite{gamevalueLip}, a regularized tug-of-war game was constructed to provide a Lipschitz continuous game value.

In particular, tracking the constants in their arguments yields $H_R(p,d)\leq\exp\left(C_R\left(\frac{d}{p-1}\right)^2\right)$ for $p<d$. This already improves the dependence on the dimension in the Moser-iteration bound. % of the Moser iteration concerning the dependence of the constant on dimension. 
Next, we state a refinement of the bounds from \cite{LPSharnack} and \cite{LPregular}.

%However, their result can be refined to prove $H_R(p,d) \le \exp\left( C_R \, \frac{d}{p-1}\log\frac{d}{p-1}\right)$ for $1<p<d$.
%We present a probabilistic proof of the Harnack inequality for $p$-harmonic functions for $p>1$, which improves the method in \cite{LPSharnack,LPregular}. %The key ingredient is   an interior H\"older estimate for $p$-harmonic functions via tug-of-war. 
%We obtain better dependence on $p$ and the dimension $d$. 
%of the constants in the H\"older and Harnack inequalities. 

\begin{theorem}\label{thm.harnack}
%[Harnack inequality]

        Assume $1<p\leq \frac{d+1}{2}$ and $R>1$. Let $u$ be a positive $p$-harmonic function in $B_R(0)$. Then $u$ satisfies the Harnack inequality \eqref{eq.harnack} with
        \begin{equation*}
            H_R(p,d) \le \exp\left(  \, \frac{C_Rd}{p-1}\log\frac{C_Rd}{p-1}\right).
        \end{equation*}
         Here $C_R$ depends only on $R$. 
    \end{theorem}
%\begin{remark} 
The proof   is given in Section \ref{sec.tugofwar}.
See Corollary \ref{cor.quantitative1} for bounds on $C_R$ and an improved estimate for large $R$. %Although the dependence on   $p$ and $d$ is not emphasized in \cite{LPSharnack,LPregular}, their arguments yield $H_R(p,d) \le  \exp(C_R(p) d^2)$ for the constant in \eqref{eq.harnack}.
%\end{remark} 
\begin{remark}
    Although the result is stated for $p\leq\frac{d+1}{2}$, a similar argument also works when $p\in(\frac{d+1}{2},\infty)$. The only difference is the choice of fundamental solutions when applying Lemma \ref{lem.holdertoharnack}. Since we are interested in the limiting behavior as $p\downarrow 1$ and $d\rightarrow\infty$, we do not cover this case here.
\end{remark}

  In Section \ref{sec.planar}, we present a direct proof of the Harnack inequality in dimension 2 that uses planar geometry instead of   the H\"older estimate and gives better dependence of $H_R(p,2)$ on $p$: 

\begin{proposition}[Restated in Proposition \ref{prop.planar}]\label{prop.planar0}
    Let $u$ be a positive $p$-harmonic function on $B_R(0)\subset\R^2$, where $R>1$ and $1<p<2$. Then $u$ satisfies the Harnack inequality \eqref{eq.harnack} with $H_R(p,2)\leq\exp\left(\frac{C_R}{p-1}\right)$.
\end{proposition}

%In Section \ref{sec:discussion}, we discuss the Harnack constant obatined from different analytic and probabilistic methods. In particular, the classical Moser iteration yields a larger upper bound $H_R(p,d) \le \exp(\exp(C_R(p) d))$. However, there is another direct approach from the log-gradient estimate that shows $H_R(p,d)\leq \exp\left(\frac{Cd}{p-1}\right)$.

\subsection{Log-gradient estimate}

Kotschwar and  Ni \cite{loggradient} developed another analytic approach, based on a log-gradient estimate, by adapting the ideas of Cheng-Yau \cite{ChengYau} and Li-Yau \cite{LiYau} to the $p$-harmonic setting.  The following inequality is a direct corollary of \cite[Theorem 1.1]{loggradient}:
\begin{proposition}[Kotschwar-Ni]\label{prop.loggradient}
    There exists a constant $c_p$ depending only on $p$ such that for any positive $p$-harmonic function $u$ in $B_R(0)\subset\Rd$, we have
    \begin{equation*}
        \sup_{B_{\frac{R}{2}}(0)}|\nabla\log u|\leq \frac{c_pd}{R(p-1)},
    \end{equation*}
    where $c_p$ remains bounded as $p\downarrow 1$.
\end{proposition}

Suppose $u$ is positive and $p$-harmonic on $B_2(0)$. The above estimate directly implies
\begin{equation*}
    \sup_{x\in B_1(0)}\log u(x)-\log u(0)\leq \frac{c_pd}{p-1}.
\end{equation*}
This shows $H_2(p,d)\leq \exp\left(\frac{c_pd}{p-1}\right)$. The details of this method are discussed in Section \ref{sec.loggradient}.

Section~\ref{sec.tugofwar} develops the probabilistic approach, and Section~\ref{sec.loggradient} explains the analytic log-gradient argument. For completeness, Appendix~\ref{sec.measurability} proves that the fixed-step game used here has a Borel value satisfying the dynamic programming principle and admits Borel approximately optimal action choices. This result applies to the precise game in a ball, with fixed tug length, spherical noise, and continuous boundary payoff; its proof is independent of the regularity estimates.

\section{Harnack inequality via tug-of-war}\label{sec.tugofwar}
The tug-of-war approach to the Harnack inequality proceeds through a H\"older estimate, first developed in \cite{LPSharnack} for     $p>2$. Here we present the   general $p>1$ case with particular attention to the dependence of the Harnack constant on $d$ and $p$. We focus on the range $1<p<d$.
\subsection{H\"older estimate}
For any function $f:U\rightarrow\R$,   define the oscillation  
\begin{equation}
    \osc(f,U):=\sup_Uf-\inf_U f.
\end{equation} 

A key ingredient in the proof of Theorem \ref{thm.harnack} is the following H\"older estimate.

\begin{proposition}[Interior H\"older continuity] \label{thm.holder}
    Suppose $1<p<d$. There exist universal constants $C,\gamma>0$, such that the following holds. If $u\in C(\ov{B_1(0)})$ is $p$-harmonic (in the viscosity sense) in $ B_{1}(0)$, then, for every $0<\delta<1$, we have 
    \begin{equation}\label{eq.holder}
        \osc(u,B_\delta(0))\leq \frac{Cd}{p-1} \delta^\gamma \osc(u,B_1(0)).
    \end{equation}
    In particular, we can take $\gamma=\frac{1}{5}$.
\end{proposition}
The explicit dependence on $p$ and $d$ is important for our purposes, while the value of the H\"older exponent $\gamma$ is not.
We note that $u$ is, in fact, locally of class $C^{1,\alpha}$, as proved by Uraltseva \cite{Ural'cevainteriorC1alpha} for $p>2$ and later by Lewis \cite{LewisinteriorC1alpha} and DiBenedetto \cite{DibenedettointeriorC1alpha} for all $p>1$.

To prove this result in a probabilistic way, we first recall the relationship between   $p$-harmonic functions in the unit ball $B_{1}(0)$ and  tug-of-war with noise. See \cite[Section 1.1]{PS08} or \cite{Parviainenbook} for details.
    
    Given   boundary values $F:\partial B_1(0)\rightarrow \R$
    and $\epsilon>0$, a two-player zero-sum game  with starting point $X_0=x$  is defined as follows. At step $n\ge 1$ a fair coin is tossed and the winner determines the next position $X_n$ as follows. If $B(X_{n-1},(\rho+1)\epsilon)$ intersects  $\partial B_1(0)$,  the winner chooses $X_n$ from this intersection; we then define $\tau_0=n$,  and the game terminates.  Otherwise, the winner chooses a vector $v_n\in \partial B_\epsilon(0)$, and a random noise $w_n$  is drawn uniformly from the $(d-2)$-dimensional sphere $ \{w\in\Rd: \, w \perp v_n, \, \,  \|w\|_2=\rho\Vert v_n\Vert_2\}$, where $\rho=\sqrt{\frac{d-1}{p-1}}$. In this case, $X_n:=X_{n-1}+v_n+w_n$.   
    %Player I always tries to maximize the value $F(X_\tau)$, while player II tries to minimize it. 
    At time $\tau_0$, player II pays the amount $F(X_{\tau_0})$ to player I.
   The value for player I of this game   
   starting from $x$ is  $u^1_\epsilon(x):=\sup_{S_I}\inf_{S_{II}}\mathbb E_{S_I,S_{II}}[F(X_{\tau_0})\mathbf 1_{\tau_0<\infty}]$. The value for player II of this game   
   starting from $x$ is  $u^2_\epsilon(x):=\inf_{S_{II}}\sup_{S_I}\mathbb E_{S_I,S_{II}}[F(X_{\tau_0})\mathbf 1_{\tau_0<\infty}]$. We write the convergence result in \cite[Theorem 1.2]{PS08} as a lemma.
   \begin{lemma}\label{lem.convergence}
       Suppose the boundary data $F:\partial B_1(0)\rightarrow \R$ is continuous. Then as $\epsilon$ tends to $0$, the values $u_\epsilon^1(x)$ and $u_\epsilon^2(x)$ converge uniformly to the unique viscosity solution $u:B_1(0)\rightarrow \R$ of the $p$-Laplace equation with boundary values $F$.
   \end{lemma}

   \begin{remark}
       In the original formulation of Peres and Sheffield \cite{PS08}, the tug vector $v_n$ can be chosen in $\overline{B_\epsilon(0)}$. But their proof is still valid if the choice is restricted to $v_n\in \partial B_\epsilon(0)$. Indeed, the key step is to show convergence for smooth $p$-harmonic functions with nonvanishing gradient. In this setting, it is shown in  \cite[Lemma 2.3]{PS08} that a tug of length $\epsilon$ in the gradient direction, or the opposite direction, is nearly optimal for the corresponding player, regardless of the opponent's strategy. Restricting the opponent to the smaller set of strategies cannot invalidate this estimate. The proofs of \cite[Theorems 2.4 and 2.5]{PS08}  rely only on the gradient strategy and therefore carry over under the restriction. Therefore, the value of the game does not change in the limit $\epsilon\rightarrow 0$ if we restrict to those strategies that only allow one to tug distance $\epsilon$. The passage from smooth solutions with non-vanishing gradient to general viscosity solutions is purely abstract, so their Theorem 1.2 remains valid with this definition. Therefore, we use this game with $v_n\in \partial B_\epsilon(0)$ for convenience.
   \end{remark}

   \begin{remark}
       Peres and Sheffield \cite{PS08} define the upper and lower values of the game using pure Borel-measurable strategies. A pure strategy prescribes a deterministic move after every possible history of the game.

    For the proof below, however, it is convenient to describe behavioral strategies, by means of Borel probability kernels depending on the complete game history. Allowing behavioral strategies does not enlarge the best-response supremum. Indeed, because the history and action spaces are standard Borel, these kernels can be realized using an independent sequence $U=(U_n)_{n\geq1}$ of uniform random variables. Let $S_I$ be a behavioral strategy. For each realization $r$ of $U$, the resulting contingent plan $S_I^r$ is a pure Borel strategy. Consequently, for every fixed strategy $S_{II}$,
\begin{equation*}
    \E_{S_I,S_{II}}[f]=\int \E_{S_I^r,S_{II}}[f]\,\d\P_U(r)\leq \sup_{\widetilde S_I\text{ pure}}\E_{\widetilde S_I,S_{II}}[f].
\end{equation*}
Thus every payoff lower bound obtained using the behavioral counter-strategy also holds for the supremum over pure strategies. Taking
\begin{equation*}
    f=F(X_{\tau_0})\mathbf 1_{\{\tau_0<\infty\}}
\end{equation*}
and then taking the infimum over pure strategies $S_{II}$, we obtain
\begin{equation*}
    u_{\ep,\text{behavioral}}^2\leq u_{\ep,\mathrm{pure}}^2,
\end{equation*}
where $u_{\ep,\mathrm{pure}}^2=u^2_{\ep}$ is the game value defined above using pure strategies, and $u^2_{\ep,\text{behavioral}}$ is the game value defined using behavioral strategies. Using a symmetric argument, we obtain
\begin{equation}
    u_{\ep,\text{pure}}^1\leq u_{\ep,\mathrm{behavioral}}^1\leq u_{\ep,\text{behavioral}}^2\leq u_{\ep,\mathrm{pure}}^2.
\end{equation}

Therefore, the value defined by behavioral strategies has the same convergence as the value defined by pure strategies.

This is the standard purification argument, closely related to Kuhn's theorem \cite{Kuhnfinite} for games with perfect recall and to Aumann's extension \cite{Aumanninfinite} to infinite extensive games.
   \end{remark}

   We also use the following lemma, which is the generalization of \cite[Lemma 2.3]{LPSharnack}.

   In the following arguments, $\epsilon$ is small enough that $(1+\rho)\epsilon<1$. We extend both game values to the boundary by setting $u^i_\epsilon=F$ on $\partial B_1(0)$, for $i=1,2$.

   \begin{lemma} \label{lem.kappastrategy}
       Fix the continuous boundary function $F:\partial B_1(0)\rightarrow \R$. For any $\kappa>0$, there exists a pure Borel strategy $S_I^\kappa$ of player I, depending only on the current position and the move number, such that, for every starting point $x$, any strategy $S_{II}$ of player II, and any almost surely finite stopping time $\sigma$ with $\sigma\leq \tau_0$, we have
       \begin{equation}\label{eq.kappastrategy1}
           u_\ep^1(x)\leq \E^x_{S_{I}^\kappa,S_{II}}[u_\ep^1(X_\sigma)]+\kappa.
       \end{equation}
       Similarly, there exists a pure Borel strategy $S_{II}^\kappa$ of player II, depending only on the current position and the move number, such that, for any strategy $S_I$ and any almost surely finite stopping time $\sigma$ with $\sigma\leq \tau_0$, we have
       \begin{equation}
           u_\ep^2(x)\geq \E^x_{S_{I},S_{II}^\kappa}[u_\ep^2(X_\sigma)]-\kappa.
       \end{equation}
   \end{lemma}

   \begin{proof}
       We prove \eqref{eq.kappastrategy1} as an example. We may continue the game after $\tau_0$ by freezing the token at the boundary and remove the assumption $\sigma\leq \tau_0$. 
       
       For $x\in B_1(0)$, let $A_\epsilon(x)$ denote the space of (pure) choices of the winning player at position $x$. For $a\in A_\epsilon(x)$, we define
       \begin{equation*}
           Q_\epsilon (x,a)=\E^x[u_{\epsilon}^1(X_1)|\text{ the winning player chooses }a].
       \end{equation*}
       Proposition~\ref{prop.borelvalue}, proved independently in Appendix~\ref{sec.measurability}, gives $u^1_\epsilon=u^2_\epsilon$ and the dynamic programming identity. In particular,
       \begin{equation}\label{eq.onesidedpp}
           u_\ep^1(x)\leq \frac{1}{2}\left(\sup_{a\in A_\ep(x)}Q_\epsilon(x,a)+\inf_{a\in A_\ep(x)}Q_\epsilon(x,a)\right).
       \end{equation}
       The same proposition supplies Borel selectors $a^+_\delta$ satisfying \eqref{eq.borel.selectors}. At time $n$, given the position $X_{n-1}$, choose $a_n=a^+_{\kappa2^{-n}}(X_{n-1})$. This is a pure Borel strategy and satisfies
       \begin{equation*}
           Q_\epsilon(X_{n-1},a_n)\geq\sup_{a\in A_\epsilon(X_{n-1})} Q_\epsilon(X_{n-1},a)-\kappa2^{-n}.
       \end{equation*}

       For any strategy $S_{II}$, it follows that
       \begin{align*}
           \E^x_{S_I^\kappa,S_{II}}[u_\ep^1(X_n)\mid\F_{n-1}]
           &\geq \frac12\sup_{a\in A_\ep(X_{n-1})}Q_\ep(X_{n-1},a)-\kappa2^{-n-1}
           \\&\quad+\frac12\inf_{a\in A_\ep(X_{n-1})}Q_\ep(X_{n-1},a).
       \end{align*}
       This, together with \eqref{eq.onesidedpp}, shows that 
       \begin{equation*}
           M_n=u_\ep^1(X_n)-\kappa2^{-n-1}
       \end{equation*}
       is a submartingale. The optional stopping theorem for this bounded submartingale gives the desired result.
   \end{proof}
   
   %In , it is shown that both $u^1_\epsilon (x)$ and $u^2_\epsilon(x)$ converge to the unique $p$-harmonic function $u(x)$ with boundary data $F$ as $\epsilon\rightarrow 0$. 
   %Since we finally take the limit $\epsilon\rightarrow 0$, using either value is not essential to the proof. In the following context, we write $u_\epsilon(x)$ for either of the two values. 

\begin{proof}[Proof of Proposition \ref{thm.holder}]
    For any two points $x,y\in B_{\delta}(0)$, consider two games starting at $x$ and $y$ with boundary data $F=u$ on $\partial B_1(0)$. Denote their positions by $(X_n)$ and $(Y_n)$. Fix $\kappa>0$ and choose the Borel Markov strategies $S_I^\kappa,S_{II}^\kappa$ from Lemma~\ref{lem.kappastrategy}, for player I in the first game and player II in the second. Write $U_{n+1}$ and $V_{n+1}$ for their prescribed moves at $X_n$ and $Y_n$.

    Let $\F_n$ be the $\sigma$-field generated by the full histories of both games, including all coin tosses and other randomization revealed through time $n$. The counter-moves below are Borel functions of this joint history. We couple the next coins and noises so that, conditionally on $\F_n$, each marginal coin is fair and each marginal noise has its prescribed distribution given the chosen tug. All prescriptions are needed only before the stopping time $\tau$ defined below and are extended afterwards by fixed Borel legal rules.

    In the following, we use the notation
    \begin{equation}
        \hat v:=\frac{v}{|v|}
    \end{equation}
    for $v\in\Rd\setminus\{0\}$. We also set
    \begin{equation}
        Z_n:=X_n-Y_n.
    \end{equation}

    In the game starting from $x$, define the counter-strategy $S_{II}^*=S_{II}^*(S_I^\kappa)$ as follows. Given $X_n,Y_n$ and the prescribed move $U_{n+1}$, it chooses
    \begin{align*}
        \begin{cases}
            -U_{n+1},\ &\text{if}\  \widehat{Z_n}\bigcdot\widehat{U_{n+1}}>\cos \theta_0;
            \\ -\epsilon\widehat{Z_n},\ &\text{otherwise}.
        \end{cases}
    \end{align*}
    Here $\theta_0$ is some constant to be determined later. See Figure \ref{fig.strategy}.
    
\begin{figure}
    \centering
    \begin{minipage}{.49\linewidth}
        \centering
        \includegraphics[height=.6\linewidth]{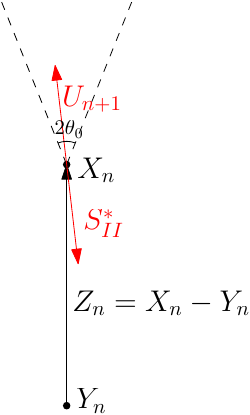}
    \end{minipage}
    \begin{minipage}{.49\linewidth}
        \centering
        \includegraphics[height=.6\linewidth]{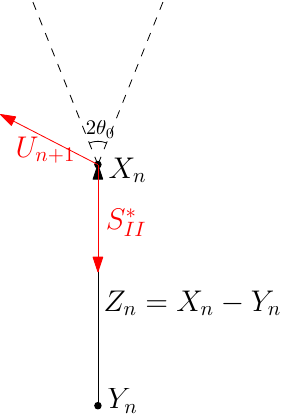}
    \end{minipage}
    \caption{Left: when $U_{n+1}$ points nearly in the direction of $\widehat{Z_n}$, the strategy $S_{II}^*$ attempts to cancel this intended move. Right: otherwise, the strategy just prescribes a move toward $Y_n$.}
    \label{fig.strategy}
\end{figure}
    Similarly, in the game starting from $y$, define $S_I^*=S_I^*(S_{II}^\kappa)$. Given $X_n,Y_n$ and the prescribed move $V_{n+1}$, it chooses
    \begin{align*}
        \begin{cases}
            -V_{n+1},\ &\text{if}\  -\widehat{Z_n}\bigcdot\widehat{V_{n+1}}>\cos \theta_0;
            \\ \epsilon\widehat{Z_n},\ &\text{otherwise}.
        \end{cases}
    \end{align*}
    
    We now define the coupled process. The key step is to define an appropriate coupling of the noise variables. We separate the discussion into two cases.

    \textbf{Aligned case: }During the $(n+1)$-th step, if both players' choices are $\theta_0$-aligned with $\widehat{Z_n}$ and $-\widehat{Z_n}$ respectively, then we define a special coupling of the two processes. More precisely, if
    %both of the two following cases happen
    \begin{equation}\label{eq.unbiasedcondition}
    \widehat{Z_n}\bigcdot\widehat{U_{n+1}}>\cos \theta_0 \, \quad \, \text{and} \, \, \quad  
         -\widehat{Z_n}\bigcdot\widehat{V_{n+1}}>\cos \theta_0,
    \end{equation}
    then we use a single coin toss to determine which player moves in each game. If the result is heads, then player I moves in the process $(X_n)$ and player II moves in the process $(Y_n)$. We also couple the corresponding noise variables. By the definition of the game, in each marginal process,  the noises must be independent of the coin toss. Note that in $(X_n)$, the noise is always orthogonal to $U_{n+1}$, while in $(Y_n)$, the noise is always orthogonal to $V_{n+1}$. By the assumption, in this case, the angle  between either of the two normal vectors and $\pm Z_n$ is less than $\theta_0$. Therefore, we can find two rotations $S,S'\in SO(d)$ such that
    \begin{itemize}
        \item $S\hat{Z}_n=\hat{U}_{n+1}$ and $S'\hat Z_n=-\hat V_{n+1}$.
        \item $S$ and $S'$ are $2$-dimensional rotations of angles $\theta_1$ and $\theta_2$ respectively, where $\theta_1,\theta_2\in[0,\theta_0]$.
    \end{itemize}
    We take the minimal rotations in the indicated planes, with the identity when the angle is zero; these depend continuously on the directions in the aligned case. To couple the noises, draw $B_{n+1}$ with the prescribed law perpendicular to $U_{n+1}$, independently of the new coin. Then $S'S^{-1}B_{n+1}$ has the required law perpendicular to $V_{n+1}$. This gives a Borel coupling kernel.

     We claim that for $\theta_0=\frac{1}{10}$, 
    \begin{equation}\label{eq.normalgeqtotal}
        \E\left[\left.\left(\left(S'S^{-1} B_{n+1}-B_{n+1}\right)\bigcdot \widehat{Z_{n}}\right)^2\right|\F_{n}\right]\geq \frac{3}{4}\E\left[\left.\left|S'S^{-1} B_{n+1}-B_{n+1}\right|^2\right|\F_{n}\right].
    \end{equation}
    We state it as an abstract lemma.

    \begin{lemma}
        Let $Z,U,V\in \Rd\setminus\{0\}$. Suppose that 
        \begin{equation*}
            \widehat{Z}\bigcdot\widehat{U}>\cos \theta_0 \, \quad \, \text{and} \, \, \quad  
         -\widehat{Z}\bigcdot\widehat{V}>\cos \theta_0.
        \end{equation*}
        Let $S,S'\in SO(d)$ such that
        \begin{itemize}
        \item $S\hat Z=\hat U$ and $S$ fixes the space $\mathrm{span}\{Z,U\}^\perp$. Similarly, $S'\hat Z=-\hat V$, and $S'$ fixes the space $\mathrm{span}\{Z,V\}^\perp$.
        \item $S$ and $S'$ are $2$-dimensional rotations with angles $\theta_1$ and $\theta_2$ respectively, where $\theta_1,\theta_2\in[0,\theta_0]$.
    \end{itemize}
    Let $B$ be a bounded random vector in $\mathrm{span}\{U\}^\perp$. Suppose that the law of $B$ is invariant under any orthogonal transformation that fixes $\mathrm{span} \{U\}$. If $\theta_0\leq\frac{1}{10}$, then
    \begin{equation*}
        \E\left[\left(\left(S'S^{-1} B-B\right)\bigcdot \widehat{Z}\right)^2\right]\geq \frac{3}{4}\E\left[\left|S'S^{-1} B-B\right|^2\right].
    \end{equation*}
    \end{lemma}
    \begin{proof}
        Let $W=\mathrm{span}\{Z,U,V\}$ and $P_W$ be the orthogonal projection onto $W$. Note that $P_WB$ is a bounded random vector in $\{x\in W:\ x\perp U\}$ and its law is invariant under any orthogonal transformation of $W$ that preserves $\mathrm{span}\{U\}$. Also note that $S'S^{-1}P_{W^\perp}B=P_{W^\perp}B$. Therefore, we can, without loss of generality, replace $\Rd$ by $W$ and assume $d\leq 3$. The result is trivial when $d=1$ and easy to check when $d=2$. Therefore, we focus on the $d=3$ case. 

         We observe that conditional on $|B|$, the vector $B$ is uniformly distributed on the $(d-2)$-dimensional sphere perpendicular to $U$ with radius $|B|$. Therefore, all the terms can be calculated explicitly. Let $\psi\in[0,2\pi]$ denote the angle between the normal vectors of the two rotations. An elementary geometric calculation then gives
    \begin{align*}
        \E\left[\left.\left(\left(S'S^{-1} B-B\right)\bigcdot \widehat{Z}\right)^2\right||B|\right]&=\frac{|B|^2}{2\pi}\int_{0}^{2\pi}|\cos\phi\sin\theta_1-\cos(\phi+\psi)\sin\theta_2|^2\d\phi
        \\&=\frac{|B|^2}{2}\left((\sin\theta_1-\cos\psi\sin\theta_2)^2+(\sin\psi\sin\theta_2)^2\right)
        \\&=\frac{|B|^2}{2}\left(\sin^2\theta_1+\sin^2\theta_2-2\sin\theta_1\sin\theta_2\cos\psi\right).
    \end{align*}
        Similarly, for the part projected to $\widehat{Z}^\perp$, we have 
    \begin{align*}
        &\E\left[\left.\left|P_{\mathrm{span}\{Z\}^\perp}\left(S'S^{-1} B-B\right)\right|^2\right||B|\right]
        \\=&\frac{|B|^2}{2\pi}\int_0^{2\pi} (\cos\phi\cos\theta_1-\cos(\phi+\psi)\cos\theta_2\cos\psi-\sin(\phi+\psi)\sin\psi)^2\d\phi
        \\&+\frac{|B|^2}{2\pi}\int_0^{2\pi} (\sin\phi+\cos(\phi+\psi)\cos\theta_2\sin\psi-\sin(\phi+\psi)\cos\psi)^2\d\phi
        \\=& \frac{|B|^2}{2}\left((\cos\theta_1-\cos\theta_2\cos^2\psi-\sin^2\psi)^2+ (\cos\theta_2-1)^2\cos^2\psi\sin^2\psi\right)
        \\&+\frac{|B|^2}{2}\left((\cos\theta_2-1)^2\cos^2\psi\sin^2\psi+ (\cos\theta_2-1)^2\sin^4\psi\right)
        \\=&\frac{|B|^2}{2}\left((1-\cos\theta_1)^2+(1-\cos\theta_2)^2-2(1-\cos\theta_1)(1-\cos\theta_2)\cos^2\psi\right).
    \end{align*}
    
    As long as $\theta_0\leq\frac{1}{10}$, a direct calculation shows
    \begin{equation*}
        \E\left[\left.\left(\left(S'S^{-1} B-B\right)\bigcdot \widehat{Z}\right)^2\right||B|\right]\geq 3\E\left[\left.\left|P_{\mathrm{span}\{Z\}^\perp}\left(S'S^{-1} B-B\right)\right|^2\right||B|\right],
    \end{equation*}
    which proves the desired result.
        
    \end{proof}

    Moreover, the tug increments in the two games are nearly parallel to $\widehat{Z_n}$. If we denote by $v^X_{n+1}$ and $v^Y_{n+1}$  the tug vectors chosen in the games $(X_n)$ and $(Y_n)$, we have $|(v^X_{n+1}-v^Y_{n+1})\cdot \widehat{Z_{n}}|^2\geq \frac{3}{4}|(v^X_{n+1}-v^Y_{n+1})|^2$ almost surely. 
    
    Set $D_{n+1}=Z_{n+1}-Z_{n}$. We can combine the two observations above and observe that the tug increments and the noise increments are uncorrelated (so that the cross term vanishes) to conclude that
    \begin{equation}\label{eq.verticalgeqtotal}
        \E \left[\left.\left|D_{n+1}\bigcdot \widehat {Z_{n}}\right|^2\right|\F_{n}\right]\geq \frac{3}{4}\E \left[|D_{n+1}|^2|\F_{n}\right].
    \end{equation}
    %In particular, their difference $D_n$ is always bounded by $\epsilon \rho\bigcdot2\theta_0$. We analyze the process $Z_n=Y_n-X_n$. Note that
    %\begin{equation*}
        %Z_{n}-Z_{n-1}=\begin{cases}
            %V_n-U_n+D_n,\ \text{with probability}\  \frac{1}{2};
            %\\ U_n-V_n+D_n,\ \text{with probability}\  \frac{1}{2}.
        %\end{cases}
    %\end{equation*}
    Note that in this case, by the definition of the strategies, we have
    \begin{equation}\label{eq.unbias}
        \E [D_{n+1}|\F_{n}]=0.
    \end{equation} 
    %\begin{equation}
        %\E \left[\left.\left|(Z_n-Z_{n-1})\bigcdot \frac{Z_{n-1}}{|Z_{n-1}|}\right|^2\right|\F_{n-1}\right]\geq (2\epsilon\cos\theta_0)^2-(2\epsilon \rho\theta_0)^2\geq 3\epsilon^2
    %\end{equation}
    %for sufficiently small $\theta_0=\min\left\{\frac{1}{4\rho},\frac{1}{10}\right\}$. Similarly, we can compute
    %\begin{equation}
        %\E [\left|Z_n-Z_{n-1}\right|^2|\F_{n-1}]\leq (2\epsilon)^2+(2\epsilon \rho\theta_0)^2\leq 5\epsilon^2.
    %\end{equation}

    We use a Taylor expansion for the function $(1+x)^{\beta}$ with $\beta=1/10$. We also fix some $\eta>0$ to be determined later, and assume that $\eta<|Z_i|\leq 2$ holds for all $i$. We have
    \begin{equation*}
        \begin{split}
            |Z_{n+1}|^{2\beta}=&|Z_{n}|^{2\beta}\left(1+2\widehat{Z_{n}}\bigcdot\frac{D_{n+1}}{|Z_{n}|}+\left\vert\frac{D_{n+1}}{|Z_{n}|}\right\vert^2\right)^{\beta}
            \\=&|Z_n|^{2\beta}\left(1+2\beta\widehat{Z_n}\bigcdot\frac{D_{n+1}}{|Z_n|}
              +\beta\left|\frac{D_{n+1}}{|Z_n|}\right|^2\right.
            \\&\hspace{55pt}\left.+2\beta(\beta-1)\left(\widehat{Z_n}\bigcdot\frac{D_{n+1}}{|Z_n|}\right)^2+O(\epsilon^3)\right).
        \end{split}
    \end{equation*}
    Here the notation $O(\epsilon^3)=O_{\eta,\rho}(\epsilon^3)$ means a remainder bounded by $C(\eta,\rho)\epsilon^3$.
    Taking a conditional expectation with respect to $\F_n$ and applying \eqref{eq.verticalgeqtotal} and \eqref{eq.unbias}, we obtain
    \begin{equation} \label{eq.marZunbiased}
    \begin{split}
                &\E[|Z_{n+1}|^{2\beta}|\F_{n}]
                \\=&|Z_{n}|^{2\beta}\left(1+\beta|Z_{n}|^{-2}\E[|D_{n+1}|^2|\F_{n}]+2\beta(\beta-1)|Z_{n}|^{-2}\E[|D_{n+1}\bigcdot\widehat{Z_{n}}|^2|\F_{n}]+O(\epsilon^3)\right)
                \\\leq&|Z_{n}|^{2\beta}-\frac{\beta}{4}|Z_{n}|^{2\beta-2}\E[|D_{n+1}|^2|\F_{n}]+O(\epsilon^3)
                \\\leq&|Z_{n}|^{2\beta}-\frac{1}{200}\epsilon^2
    \end{split}
    \end{equation}
    for $\epsilon<\epsilon_0=\epsilon_0(\eta,\rho)$. Here,  in passing from the second line to the third, we use \eqref{eq.verticalgeqtotal} and note that the particular choice of $\beta$ guarantees $1+3(\beta-1)/2\leq-1/4$. In passing from the third line to the fourth we use $\E[|D_{n+1}|^2|\F_{n}]\geq \epsilon^2$ due to the tug component.
    
    %Combining the computations above, we can do a martingale type computation. Recall that until now, we are under the assumption \eqref{eq.unbiasedcondition}. We also fix some $\eta>0$ to be determined later, and assume $\eta<|Z_i|\leq 1$ holds for all $i$.
    %\begin{equation}\label{eq.marZunbiased}
    %\begin{split}
        %&\E[|Z_n|^{2\beta}|\F_{n-1}]\\=&\E\left[\left.\left(|Z_{n-1}|^2+2 Z_{n-1}\bigcdot (Z_n-Z_{n-1})+|Z_{n}-Z_{n-1}|^2\right)^\beta\right|\F_{n-1}\right]
        %\\=&|Z_{n-1}|^{2\beta}\E\left[\left.\left(1+2 \frac{Z_{n-1}}{|Z_{n-1}|}\bigcdot \frac{Z_n-Z_{n-1}}{|Z_{n-1}|}+\frac{|Z_{n}-Z_{n-1}|}{|Z_{n-1}|^2}^2\right)^\beta\right|\F_{n-1}\right]
        %\\=& |Z_{n-1}|^{2\beta}\left(1+\beta\E\left[\left.2 \frac{Z_{n-1}}{|Z_{n-1}|}\bigcdot \frac{Z_n-Z_{n-1}}{|Z_{n-1}|}+\frac{|Z_{n}-Z_{n-1}|}{|Z_{n-1}|^2}^2\right|\F_{n-1}\right]
        %\right.\\&+\left.\frac{\beta(\beta-1)}{|Z_{n-1}|^2}\E \left[\left.\left|2(Z_n-Z_{n-1})\bigcdot \frac{Z_{n-1}}{|Z_{n-1}|}\right|^2\right|\F_{n-1}\right]+O(\epsilon^3)\right)
        %\\ \leq &|Z_{n-1}|^{2\beta}+|Z_{n-1}|^{2\beta-2}2\beta \bigcdot\E \left[\left.\left|(Z_n-Z_{n-1})\bigcdot \frac{Z_{n-1}}{|Z_{n-1}|}\right|^2\right|\F_{n-1}\right]
        %\\&+4|Z_{n-1}|^{2\beta-2}\beta(\beta-1)\bigcdot \E [\left|Z_n-Z_{n-1}\right|^2|\F_{n-1}]+O(\epsilon^3)
        %\\ \leq & |Z_{n-1}|^{2\beta}-\frac{1}{100}\epsilon^2,
   % \end{split}
    %\end{equation}
    %as long as $\epsilon<\epsilon_0=\epsilon_0(\eta)$. 

    Meanwhile, in this case, the processes $X_n$ and $Y_n$ are unbiased, and we have
    \begin{equation}\label{eq.marXYunbiased}
        \begin{split}
            \E[|X_{n+1}|^{2}|\F_{n}]\leq|X_{n}|^2+\epsilon^2+\rho^2\epsilon^2;
            \\ \E[|Y_{n+1}|^{2}|\F_{n}]\leq|Y_{n}|^2+\epsilon^2+\rho^2\epsilon^2.
        \end{split}
    \end{equation}
    
    Set 
    \begin{equation}\label{eq.defM}
        M_n=|X_{n}|^2+|Y_n|^2+C|Z_n|^{2\beta},
    \end{equation}
    where $C>400(\rho^2+1)$ is a constant to be chosen later.
    
    Combining \eqref{eq.marZunbiased} and \eqref{eq.marXYunbiased}, we have
    \begin{equation}\label{eq.martingaleM1}
        \E[M_{n+1}|\F_{n}]\leq M_{n}
    \end{equation}
    under the condition \eqref{eq.unbiasedcondition} and $|Z_n|>\eta$.

    \textbf{Unaligned case:} %On the other hand, if the condition \eqref{eq.unbiasedcondition} is not satisfied, we keep the coins and noises independent. In this case $|Z_{n}|$ has a negative drift. In at least one of the two games, one player is pushing directly to decrease $|Z_n|$, while his opponent is $\theta_0$-away from directly increasing $|Z_n|$. This will create a negative drift of $|Z_n|$. More precisely, as long as $\eta<|Z_{n}|\leq 2$, we have that
If condition \eqref{eq.unbiasedcondition} fails, we take the coin tosses and noise variables in the two games to be independent. In this case, $|Z_n|$ has a negative drift. Indeed, in at least one of the two games, one player moves directly so as to decrease $|Z_n|$, while the opponent’s move has an angle of at least $\theta_0$ with the direction of maximal increase. Consequently, $|Z_n|$ has a negative one-step drift. 

    \begin{equation}
    \begin{split}
        &\E[|Z_{n+1}|^{2\beta}|\F_{n}]
        \\=&|Z_{n}|^{2\beta}\left(1+\beta\E\left[\left.2 \widehat{Z_{n}}\bigcdot \frac{D_{n+1}}{|Z_{n}|}\right|\F_{n}\right]+O(\epsilon^2)\right)
        \\\leq &|Z_{n}|^{2\beta}-2\beta|Z_n|^{2\beta-1}\frac{(1-\cos\theta_0)}{2}\epsilon +O(\epsilon^2)
        \\ \leq & |Z_{n}|^{2\beta} -\frac{1}{4000}\epsilon
    \end{split}
    \end{equation}
    for $\epsilon<\epsilon_0(\eta,\rho)$.

    On the other hand, we have
    \begin{equation}
        \begin{split}
            \E[|X_{n+1}|^2|\F_{n}]\leq |X_{n}|^2+2(\rho+1)\epsilon +O(\epsilon^2)\leq |X_{n}|^2+3(\rho+1)\epsilon,
        \end{split}
    \end{equation}
    for $\epsilon<\epsilon_0(\eta,\rho)$. The analogous estimate holds for $|Y_n|^2$. In particular, as long as $|Z_{n}|>\eta$ we have 
    \begin{equation}\label{eq.martingaleM2}
        \E[M_{n+1}|\F_{n}]\leq M_{n},
    \end{equation}
    for $C\geq24000(\rho+1)$. Recall the definition of $M_n$ in \eqref{eq.defM}.

    Now we define the stopping time 
    \begin{equation}
      \begin{split}
        \tau=\inf\{n\geq0:\ &|Z_n|\leq\eta\ \text{or}\\
        &\min\{\dist(X_n,\partial B_1),\dist(Y_n,\partial B_1)\}<(\rho+1)\epsilon\}.
      \end{split}
    \end{equation}
    Equations \eqref{eq.martingaleM1} and \eqref{eq.martingaleM2} show that $M_{n\wedge\tau}$ is a (bounded) supermartingale with respect to $\F_n$, with the choice $C=24000(\rho^2+1)\ $. Also, since $|Z_{n\wedge\tau}|^{2\beta}+(\frac{1}{4000}\epsilon\wedge\frac{1}{200}\epsilon^2)(n\wedge\tau)$ is a supermartingale, the optional stopping theorem implies that $\tau$ is almost surely finite. Then we apply the optional stopping theorem to $(M_{n\wedge \tau})$ to obtain 
    \begin{equation*}
        \E[M_\tau]\leq M_0\leq C(2\delta)^{2\beta}+2\delta^2\leq C'\delta^{2\beta}.
    \end{equation*}
    Note that $M_\tau$ is nonnegative and therefore, 
    \begin{equation}
        \P[|X_\tau|^2+|Y_\tau|^2\geq (1-(\rho+1)\epsilon)^2]\leq \P[M_\tau\geq (1-(\rho+1)\epsilon)^2] \leq \frac{C'\delta^{2\beta}}{(1-(\rho+1)\epsilon)^2}.
    \end{equation}

    Now we turn to the value of the game. %We apply \cite[Lemma 2.3]{LPSharnack}. Although \cite[Lemma 2.3]{LPSharnack} is proved for $p>2$, the same argument applies in the general case $p>1$. Therefore,
    %\begin{equation}
    %    u^1_\epsilon(x)\leq \sup_{S_I}\E_{S_I,S_{II}^*}[u^1_\epsilon (X_\tau)]
    %\end{equation}
    %and
    %\begin{equation}\label{eq.martgeq}
   %     u^2_\epsilon(y)\geq \inf_{S_{II}}\E_{S_I^*,S_{II}}[u^2_\epsilon (Y_\tau)].
    %\end{equation}
    The stopping time $\tau$ is a stopping time for the joint filtration $\F_n$. Since each marginal conditional transition law is unchanged by the coupling, the DPP computation in Lemma~\ref{lem.kappastrategy} applies conditionally on $\F_n$. Thus
    \[
       u^1_\epsilon(X_{n\wedge\tau})-\kappa2^{-(n\wedge\tau)-1}
       \quad\text{and}\quad
       u^2_\epsilon(Y_{n\wedge\tau})+\kappa2^{-(n\wedge\tau)-1}
    \]
    are respectively a bounded submartingale and a bounded supermartingale in this filtration. Optional stopping therefore gives
    \[
       u^1_\epsilon(x)\leq\E[u^1_\epsilon(X_\tau)]+\kappa,
       \qquad
       u^2_\epsilon(y)\geq\E[u^2_\epsilon(Y_\tau)]-\kappa.
    \]
    Under the two coupled strategy pairs described above, we also have
    \begin{align*}
        &\E [u^1_\epsilon (X_\tau)-u^2_\epsilon (Y_\tau)]
        \\\leq &\E\left[(u^1_\epsilon(X_\tau)-u^2_\epsilon(Y_\tau))\mathbf1_{\{|X_\tau-Y_\tau|\leq\eta\}}\right]
        \\&+\P[|X_\tau-Y_\tau|>\eta]\left(\sup_{B_1}u^1_\ep-\inf_{B_1}u^2_\ep\right)
        \\ \leq& \sup_{a,b\in B_1,\ |a-b|\leq\eta} |u^1_\ep (a)-u^2_\epsilon(b)|+\frac{C'\delta^{2\beta}}{(1-(\rho+1)\epsilon)^2}\left(\sup_{B_1}u^1_\ep-\inf_{B_1} u^2_\ep\right)_+.
    \end{align*}
    Using Lemma \ref{lem.kappastrategy}, and sending $\kappa$ to zero, we conclude that
    \begin{equation*}
        u^1_\epsilon(x)-u^2_\epsilon(y)\leq \sup_{a,b\in B_1,\ |a-b|\leq\eta} |u^1_\ep (a)-u^2_\epsilon(b)|+\frac{C'\delta^{2\beta}}{(1-(\rho+1)\epsilon)^2}\left(\sup_{B_1}u^1_\ep-\inf_{B_1} u^2_\ep\right)_+.
    \end{equation*}
    Since $u^1_\epsilon$ and $u^2_\epsilon$ converge uniformly to $u$ by Lemma \ref{lem.convergence}, we can take the limit $\epsilon\rightarrow 0$ to obtain
    \begin{equation*}
        u(x)-u(y)\leq \sup_{a,b\in B_1,\ |a-b|\leq\eta} |u (a)-u(b)|+C'\delta^{2\beta}\left(\sup_{B_1}u-\inf_{B_1} u\right).
    \end{equation*}
    Letting $\eta\rightarrow 0$ and using the uniform continuity of $u$, we obtain
    \begin{equation}
        u(x)-u(y)\leq C'\delta^{2\beta}\left(\sup_{B_1}u-\inf_{B_1} u\right).
    \end{equation}
    This inequality holds for all $x,y\in B_\delta(0)$. Therefore  Proposition \ref{thm.holder} is proved.
\end{proof}
Note that this type of coupling was already used in \cite{LPregular}. Here we refine the proof and give explicit dependence of constants on $p$ and $d$. 

\subsection{Proof of the Harnack inequality}
Combining the H\"older estimate with the argument in \cite{LPSharnack} yields a Harnack inequality. Here we modify the argument to obtain sharper constants.

\begin{lemma}\label{lem.holdertoharnack}
    Let $u$ be a positive continuous function on $B_5(0)\subset \R^d$ with $u(0)=1$. Suppose that there exist     $A\geq 1$ and   $\gamma>0$ such that for all $0<r<R\leq 1$ and $|x|< 2$,
    \begin{equation}\label{eq.oscilation}
        \mathrm{osc}(u,B_r(x))\leq A\left(\frac{r}{R}\right)^\gamma \mathrm{osc}(u,B_R(x)) \,.
    \end{equation}
   Assume also that there exist $C>0$ and $\lambda>0$ such that  for all $0<r\leq 1$  and $|x|< 2 -r$,
    \begin{equation}\label{eq.mildgrowth}
        \inf_{B_r(x)}u\leq C r^{-\lambda} \,.
    \end{equation}
     Then
    \begin{equation}
        \sup_{B_1(0)}u\leq 8C\exp\left(\lambda\left(\frac{1}{\gamma}\log(4A)+2\log (2\lambda)\right)\right).
    \end{equation}
\end{lemma}

\begin{proof} 
    Set $R_k=\frac{1}{k^2}$ and $\delta=\left(4A\right)^{-1/\gamma}$. Let $x_1=0$. We define inductively
    \begin{equation*}
        M_{k}=\max_{\overline{B}_{R_k}(x_k)}u
    \end{equation*}
    and
    \begin{equation*}
        x_{k+1}\in \argmax_{\overline{B}_{R_k}(x_k)}u.
    \end{equation*}

    We assume for contradiction that 
    \begin{equation*}
        M_1\geq L:=8C\exp\left(\lambda\left(\frac{1}{\gamma}\log(4A)+2\log (2\lambda)\right)\right).
    \end{equation*}
    We claim that under this assumption, one has
    \begin{equation}\label{eq.exponentialgrowth}
        M_k\geq  2^{k-1}L
    \end{equation}
    for all $k\geq2$.

    We prove \eqref{eq.exponentialgrowth} by induction. Assume inductively that \eqref{eq.exponentialgrowth} holds at $k-1$. Note that $\sum_k R_k<2$. By the condition \eqref{eq.mildgrowth} and the choice of $L$, we have
    \begin{equation*}
        \inf_{B_{\delta R_k}(x_k)}u\leq C\delta^{-\lambda}k^{2\lambda}\leq 2^{k-3}L\leq M_{k-1}/2.
    \end{equation*}
    Then we can apply \eqref{eq.oscilation} to obtain
    \begin{align*}
        M_k=\sup_{B_{R_k}(x_k)}u
        &\geq A^{-1}\delta^{-\gamma}\left(\sup_{B_{\delta R_k}(x_k)}u-\inf_{B_{\delta R_k}(x_k)}u\right)\\
        &\geq A^{-1}\delta^{-\gamma}(M_{k-1}-M_{k-1}/2)=2M_{k-1}.
    \end{align*}
    In particular, this implies $M_{k}\geq2^{k-1}L$.
    
    Note that $\displaystyle \sup_k M_k\leq \sup_{B_{2}(0)}u<\infty$. This yields a contradiction.
\end{proof}

\begin{corollary}
    Suppose $p\in(1,\frac{d+1}{2}]$. Let $u$ be a positive $p$-harmonic function on $B_5(0)$. Then $u$ satisfies the Harnack inequality:
        \begin{equation}\label{eq.harnack51}
            \sup_{B_1(0)}u\leq H_5(p,d)u(0),
        \end{equation}
        where $H_5(p,d)\leq\exp\left( \frac{C_5d}{p-1}\log\frac{C_5d}{p-1}\right) $ for some absolute constant $C_5>0$.
\end{corollary}
 \begin{proof} 

 We may assume $u(0)=1$. It suffices to check the two conditions in Lemma \ref{lem.holdertoharnack}.
 
 The hypothesis \eqref{eq.oscilation} is verified in Proposition \ref{thm.holder} with $A=\frac{C_0d}{p-1}$ and $\gamma=\frac{1}{5}$.
 
 As noted in \cite[Section 4]{LPSharnack}, comparison with fundamental solutions yields \eqref{eq.mildgrowth} with $$C=C_02^{\frac{d}{p-1}}\text{ and }\lambda=\frac{d-p}{p-1},$$
 where $C_0$ is a universal constant. For convenience, we recall the argument.  Fix $z$ and $r$ with $|z|+r\leq 2$. We may assume $0\notin B_r(z)$ since otherwise the estimate is immediate. Consider the fundamental solution 
\begin{equation*}
    v(x)=2\frac{|x-z|^{-\lambda}-3^{-\lambda}}{|z|^{-\lambda}-3^{-\lambda}}
\end{equation*}
on $B_3(z)\setminus B_r(z)$. If the inequality $u \ge v$ held on $ \partial B_r(z) $, then the maximum principle would imply that $u(0)\geq v(0)=2$, which is a contradiction. Because $p\leq (d+1)/2$, we have $\lambda\geq 1$. Therefore,  
\begin{equation*}
    \inf_{\partial B_r(z)} u(x)\leq 2\frac{r^{-\lambda}-3^{-\lambda}}{|z|^{-\lambda}-3^{-\lambda}}\leq C_02^{\lambda}r^{-\lambda},
\end{equation*}
for some universal constant $C_0$.
\end{proof}

Once we have a Harnack inequality for fixed balls $B_1\subset B_5$, we immediately obtain a Harnack inequality for two balls of general radii $B_r\subset B_R$. This deduction is standard, but we include it   for convenience. It suffices to consider two balls $B_1\subset B_R$.
\begin{corollary}\label{cor.quantitative1}
    Suppose $p\in(1,\frac{d+1}{2}]$. There exist universal constants $\gamma,C_*>0$ such that the following holds. Let $u$ be a positive $p$-harmonic function on $B_R(0)\subset\Rd$   with $u(0)=1$.  
    \begin{enumerate}

        \item\label{item.radius.small} If $1<R <5$, then $\displaystyle \sup_{B_1(0)}u\leq\exp\left(\frac{C_*d}{p-1}\log\frac{C_*d}{p-1}\log \frac{R}{R-1}\right).$
\smallskip
        \item\label{item.radius.large} If $R\geq5$, then $\displaystyle \sup_{B_1(0)}u\leq1+R^{-\gamma}\exp\left(\frac{C_*d}{p-1}\log\frac{C_*d}{p-1}\right).$
    \end{enumerate}
    \end{corollary}
\begin{proof}
     Recall that we proved $H_5(p,d)\leq\exp\left( C\frac{d}{p-1}\log\frac{Cd}{p-1}\right)$ in \eqref{eq.harnack51}. 
     \begin{enumerate}
   \item Given $x\in B_1(0)$, we write $x_0=x$ and set $d_n=\left(\frac{5}{4}\right)^{n} (R-|x_0|)$ for all $n \ge 0$. We can inductively define   $x_n=\frac{R-d_{n}}{R-d_{n-1}}x_{n-1}$ so that $|x_n|=R-d_n$. Let $N$ be the integer such that $d_{N+1}>R\geq d_N$. In particular, the procedure terminates  in at most $\log_{\frac{5}{4}}\frac{R}{R-1}\le 5 \log \frac{R}{R-1} $ steps. Moreover, the terminal point   $x_N\in B_{R/5}(0)$  satisfies the Harnack estimate $u(x_N)\leq H_5(p,d)$.

    Note that $|x_{n-1}-x_n|=d_n-d_{n-1}=d_n/5$ and $d(x_n,\partial B_R(0))=d_n$. Applying \eqref{eq.harnack51} to $B_{d_{n}/5}(x_n)\subset B_{d_n}(x_n)$, we obtain
    \begin{equation*}
        u(x_{n-1})\leq H_5(p,d)u(x_n).
    \end{equation*} 
    Iterating this inequality, we obtain
    \begin{equation}
        u(x)\leq H_5(p,d)^{10\log \frac{R}{R-1}}.
    \end{equation}
\smallskip
    \item By the H\"older oscillation estimate \eqref{eq.holder} and   \eqref{eq.harnack51}, we have
    \begin{equation*}
        \sup_{B_1(0)}u-1\leq \osc(u, B_1(0))\leq \frac{Cd}{p-1}\left(\frac{5}{R}\right)^\gamma\osc (u,B_{\frac{R}{5}}(0))\leq \frac{Cd}{p-1}\left(\frac{5}{R}\right)^\gamma H_5(p,d),
    \end{equation*}
    which implies \ref{item.radius.large} provided $C_*$ is large enough.
    \end{enumerate}
\end{proof}

\begin{remark}
    The proof of Harnack's inequality in \cite{LPSharnack} relies on Lemma 4.1 there.   If one tracks the constants in the proof, the resulting Harnack constant is $H_R(p,d)\leq\exp\left(\tilde{C}_R\left(\frac{d}{p-1}\right)^2\right)$. By choosing the sequence of radii more carefully, Lemma \ref{lem.holdertoharnack}    improves \cite[Lemma 4.1]{LPSharnack} and yields the bound stated in Theorem~\ref{thm.harnack}.
    
    In   \cite{LPregular}, Luiro and Parviainen  prove a H\"older estimate for tug-of-war games in Section 6. We improve their argument and prove an explicit bound of order $\frac{Cd}{p-1}$. However,  this refinement does not affect the order of the constant in the Harnack inequality.
\end{remark}

\subsection{A special planar proof}\label{sec.planar}

In dimension two, a trajectory that surrounds the smaller ball replaces the
H\"older estimate in the preceding argument. This gives the sharper dependence
on $p$ stated in the introduction.

\begin{proposition}[Restating Proposition \ref{prop.planar0}]\label{prop.planar}
    Let $u$ be a positive $p$-harmonic function on $B_R(0)\subset\R^2$, where
    $R>1$ and $1<p<2$. Then \eqref{eq.harnack} holds with
    $H_R(p,2)\leq\exp\bigl(C_R/(p-1)\bigr)$.
\end{proposition}
\begin{proof}
The ball-chain argument in Corollary~\ref{cor.quantitative1} reduces the proof
to $R=5$. Fix $p$, put $\alpha=1+(p-1)^{-1/2}$, and consider the game in
$B_4(0)$ with boundary payoff $F=u|_{\partial B_4(0)}$.

Choose the rectangles $Q_i$ and target edges $\Gamma_i$ in the following
table, with coordinates in units of $1/10$.
\begin{center}
\begin{tabular}{c|c|c}
$i$ & $10Q_i$ & target edge \\ \hline
1 & $[-2,2]\times[-2,18]$ & top \\
2 & $[-4,24]\times[15,21]$ & right \\
3 & $[21,27]\times[-24,23]$ & bottom \\
4 & $[-24,28]\times[-27,-21]$ & left \\
5 & $[-27,-21]\times[-28,24]$ & top \\
6 & $[-28,18]\times[21,27]$ & right \\
7 & $[15,21]\times[13,28]$ & bottom
\end{tabular}
\end{center}

\begin{figure}[ht]
\centering
\begin{tikzpicture}[scale=.82]
  \draw[thin] (0,0) circle (4);
  \draw[thin] (0,0) circle (1);
  \node at (0,-.55) {$B_1(0)$};
  \fill (0,0) circle (1.3pt);
  \foreach \xa/\ya/\xb/\yb in {
    -.2/-.2/.2/1.8,
    -.4/1.5/2.4/2.1,
    2.1/-2.4/2.7/2.3,
    -2.4/-2.7/2.8/-2.1,
    -2.7/-2.8/-2.1/2.4,
    -2.8/2.1/1.8/2.7,
    1.5/1.3/2.1/2.8}
    {\draw[thin] (\xa,\ya) rectangle (\xb,\yb);}
  \draw[red,very thick] (-.2,1.8)--(.2,1.8);
  \draw[red,very thick] (2.4,1.5)--(2.4,2.1);
  \draw[red,very thick] (2.1,-2.4)--(2.7,-2.4);
  \draw[red,very thick] (-2.4,-2.7)--(-2.4,-2.1);
  \draw[red,very thick] (-2.7,2.4)--(-2.1,2.4);
  \draw[red,very thick] (1.8,2.1)--(1.8,2.7);
  \draw[red,very thick] (1.5,1.3)--(2.1,1.3);
  \node[fill=white,inner sep=1pt] at (0,.85) {$Q_1$};
  \node[fill=white,inner sep=1pt] at (.65,1.8) {$Q_2$};
  \node[fill=white,inner sep=1pt] at (2.4,-.4) {$Q_3$};
  \node[fill=white,inner sep=1pt] at (.2,-2.4) {$Q_4$};
  \node[fill=white,inner sep=1pt] at (-2.4,-.2) {$Q_5$};
  \node[fill=white,inner sep=1pt] at (-.6,2.4) {$Q_6$};
  \node[inner sep=1pt] at (1.75,3.15) {$Q_7$};
\end{tikzpicture}
\caption{The seven rectangles and their target edges (red). The transverse
crossings in $Q_2\cap Q_7$ close a trajectory around $B_1(0)$.}
\label{fig:rectangles}
\end{figure}
All seven rectangles are compactly contained in $B_4(0)$, and $Q_i$ is
disjoint from $\overline B_1(0)$ for $i\geq2$. Each $\Gamma_i$, $i<7$, is
at distance at least $1/10$ from $\partial Q_{i+1}$. Set
$\Gamma_0=\{0\}$ and $r_*=1/100$.

Suppose a polygonal path, starting at $0$, has successive portions
$\gamma_i\subset Q_i$ ending at points $z_i\in Q_i$ with
$d(z_i,\Gamma_i)<r_*$. The portions $\gamma_2$ and $\gamma_7$ cross $[1.5,2.1]^2$ in transverse
directions and therefore intersect. In particular, this provides a sub-curve separating $B_1(0)$ from $\partial B_4(0)$.
%Closing the intervening path gives a
%curve with winding number $-1$ about every point of $B_1(0)$.
% deform its
%portions inside their convex rectangles and its junctions inside the
%intersections of successive rectangles to the polygon with vertices
%\[
%(1.8,1.8),\ (2.4,1.8),\ (2.4,-2.4),\ (-2.4,-2.4),\
%(-2.4,2.4),\ (1.8,2.4),\ (1.8,1.8).
%\]
%These deformations avoid $\overline B_1(0)$, and the polygon surrounds
%$B_1(0)$ clockwise. Hence the curve separates $B_1(0)$ from
%$\partial B_4(0)$.

Put $\delta=r_*/10$. Let $D_i$ be the open rectangle obtained by moving each
side of $Q_i$ inward by $2\delta$, and let $G_i$ be its target edge. Set
\[
K_1=\{0\},\qquad
K_i=\{x:d(x,\Gamma_{i-1})\le r_*\},\quad 2\le i\le7.
\]
Then $K_i\Subset D_i$. Let $w_i$ be the $p$-harmonic extension of the
continuous boundary data
\[
g_i(z)=\left(1-\frac{d(z,G_i)}{\delta}\right)_+,
\]
on $\partial D_i$. Since $w_i\ge0$ and $w_i=1$ on $G_i$,
\cite[Lemma~2.2 (Cylinder lower bound)]{PWboundaryderivative}, applied after
a rigid motion with $d=2$, gives
\[
\inf_{x\in K_i}w_i(x)\ge
b_i:=\exp\!\left(-\frac{A_i}{p-1}\right),
\]
where $A_i$ depends only on the geometry.

The exterior cone condition and \cite[Proposition~1.1(ii) and
Theorem~1.2]{PS08}, with the fixed-length observation above, give uniform
convergence to $w_i$ of the lower values $w^1_{i,\epsilon}$ of the auxiliary
games in $D_i$ with payoff $g_i$. For fixed $p$ and sufficiently small
$\epsilon$, therefore,
\[
w^1_{i,\epsilon}(x)\ge \frac34 b_i
\quad(x\in K_i,\ 1\le i\le7),
\qquad \alpha\epsilon<\delta.
\]
Take $\epsilon$ smaller if necessary so that the rectangles avoid the
terminal collar of $B_4(0)$.

To ensure measurable restarts, choose finitely many points $y_{ij}\in K_i$
whose open $\delta$-balls cover $K_i$. By the definition of the lower value,
for each $y_{ij}$ we can fix a pure Borel strategy $S_{ij}$ guaranteeing
expected payoff at least $b_i/2$ against every opponent. When stage $i$ starts
at $x\in K_i$, choose the first index $j$ for which $|x-y_{ij}|<\delta$, and
put $t=x-y_{ij}$. Play the translated strategy $S_{ij}$ in $D_i+t$: subtract
$t$ from the subsequent position history before applying $S_{ij}$, and use
its prescribed tug vector. The finite Borel partition and continuous translation of histories make
this strategy Borel, and $D_i+t\subset Q_i$.

Run this strategy only until the first entrance into the boundary collar of
$D_i+t$; discard the auxiliary game's terminal boundary move. If
$Z\in\partial D_i$ denotes its virtual terminal point, then $g_i(Z)>0$ implies
$d(Z,G_i)<\delta$. Hence the actual collar entrance point $X_\sigma$ satisfies
\[
d(X_\sigma,\Gamma_i)
<\alpha\epsilon+\delta+|t|+2\delta
<5\delta<r_*.
\]
Since $0\le g_i\le1$, success has probability at least $b_i/2$,
conditionally on every entrance history: the opponent's continuation
induces a legal auxiliary strategy. Restart after each success, and use a
fixed Borel legal rule after failure or completion. These Borel stopping-time
restarts produce a pure Borel strategy in $B_4(0)$ whose stage-$i$ path,
including its linear interpolation, lies in $Q_i$.

Successive conditioning now gives
\[
\begin{aligned}
\mathbb P(\text{all seven stages succeed})
&\ge 2^{-7}\exp\!\left(-\frac{\sum_{i=1}^{7}A_i}{p-1}\right)\\
&\ge \exp\!\left(-\frac{C}{p-1}\right)=:q_p,
\end{aligned}
\]
where $C=\sum_{i=1}^{7}A_i+7\log2$, since $1<p<2$.

Put $M=\sup_{B_1(0)}u$ and $E=\{y\in B_4(0):u(y)=M\}$. On success, the surrounding curve meets $E$: otherwise $u-M$ has
constant strict sign on the curve, contradicting the maximum or minimum
principle in the enclosed component containing $B_1(0)$.
Since the steps have length less than $\alpha\epsilon$, some
token position is then within $\alpha\epsilon$ of $E$. Hence, for
\[
\tau=\inf\{n\geq0:d(X_n,E)<\alpha\epsilon\},
\]
the strategy ensures $\P[\tau<\infty]\geq q_p$.

Write $v_\epsilon=u_\epsilon^2$ for the upper value in $B_4(0)$, extended
by $F$ on its boundary, and freeze the token after termination. Fix the
strategy $S_{II}^\kappa$ from Lemma~\ref{lem.kappastrategy} and use the
strategy just constructed against it. The lemma and
$v_\epsilon\geq0$ give
\[
v_\epsilon(0)+\kappa
 \geq \E[v_\epsilon(X_{\tau\wedge n})]
 \geq \E[v_\epsilon(X_\tau)\mathbf1_{\{\tau\leq n\}}].
\]
Letting $n\to\infty$, and writing $\omega_u$ for the modulus of continuity
of $u$ on $\overline B_4(0)$, we obtain
\[
v_\epsilon(0)+\kappa
 \geq q_p M-\omega_u(\alpha\epsilon)
          -\|v_\epsilon-u\|_{L^\infty(\overline B_4(0))}.
\]
Finally let $\kappa\downarrow0$ and then $\epsilon\downarrow0$.
Lemma~\ref{lem.convergence} and uniform continuity give
$u(0)\geq q_p\sup_{B_1(0)}u$, as required.
\end{proof}

\section{Harnack constant via log-gradient estimate}\label{sec.loggradient}

The celebrated Cheng-Yau estimate is a basic tool in geometric analysis. It gives an upper bound for the logarithmic gradient of a positive harmonic function on a Riemannian manifold with Ricci curvature bounded below. See \cite{ChengYau,Yau75}. This method has also been developed in the parabolic setting and used to prove a Harnack-type inequality in \cite{LiYau}.

In \cite{loggradient}, a similar idea is used to prove the log-gradient estimate for $p$-harmonic functions. Kotschwar and Ni prove a general result on manifolds with sectional curvature bounded below. As a special case, their main result \cite[Theorem 1.1]{loggradient} implies Proposition \ref{prop.loggradient}. In this section, we summarize the result in the Euclidean case. For more details, one may refer to \cite{loggradient}.

The basic idea of this gradient estimate is to linearize the operator and adapt the Cheng-Yau argument. Let $u$ be a positive $p$-harmonic function. Let $v=-(p-1)\log u$ and $f=|\nabla v|^2$. We introduce the following linear operator
\begin{equation}
    \L (\psi)=\div \left(f^{\frac{p}{2}-1}A\nabla\psi\right)-pf^{\frac{p}{2}-1}\bracket{\nabla v,\nabla \psi},
\end{equation}
where $A=I+(p-2)\frac{\nabla v\otimes\nabla v}{f}$. Here we work at the points where $\nabla v\neq 0$. The first term here encodes the anisotropy of the $p$-Laplacian, while the second term is designed to cancel the linear gradient terms in the Bochner-type formula below.

We can write a simple Bochner-type formula for this operator. This is shown in \cite[Lemma 2.1]{loggradient}. We state a simplified version in Euclidean space.
\begin{equation}\label{eq.Bochner}
    \L(f)=2f^{\frac{p}{2}-1}|D^2 v|^2+\left(\frac{p}{2}-1\right)f^{\frac{p}{2}-2}|\nabla f|^2.
\end{equation}

Then we use a standard localization argument: we find a cut-off function $0\leq\eta\leq1$ with $\eta=1$ on $B_{R/2}$ and $\eta=0$ near $\partial B_R$ satisfying
\begin{equation*}
    |\nabla \eta|^2\leq C\eta/R^2,\ \text{and}\ D^2\eta\geq-CR^{-2}I.
\end{equation*}

We consider the truncated function $Q=\eta f$. If its maximum is zero, the desired estimate is immediate. Otherwise let $x_*$ be a point where its positive maximum is attained. Then $f>0$ in a neighborhood of $x_*$, and interior elliptic regularity justifies the differentiations there. We have
\begin{equation*}
    \nabla Q(x_*)=0,\ \L(Q)(x_*)\leq 0.
\end{equation*}
The first identity gives
\begin{equation*}
    \nabla f=-\frac{f}{\eta}\nabla \eta,
\end{equation*}
which allows us to relate the derivative of $f$ to the derivative of $\eta$ at the maximum point. 

The second inequality is the starting point for the main estimate. At $x_*$, we have
\begin{align*}
    \L (Q)=\eta\L(f)-\left(\frac{p}{2}+1\right)f^{\frac{p}{2}}\frac{\bracket{\nabla\eta,A\nabla\eta}}{\eta}+f^{\frac{p}{2}}\div(A\nabla\eta)-pf^{\frac{p}{2}}\bracket{\nabla v,\nabla\eta}.
\end{align*}
The key point is that every term involving $\nabla\eta$ can be controlled in the final calculation.

Using the $p$-Laplace equation, one may show the following Hessian estimate
\begin{equation*}
    |D^2 v|^2\geq \frac{1}{d-1}f^2+a_{p,d}\frac{|\nabla f|^2}{f}+\frac{1}{d-1}\left(\frac{p}{2}-1\right)^2\frac{\bracket{\nabla f,\nabla v}^2}{f^2}-\frac{p-1}{d-1}\bracket{\nabla f,\nabla v},
\end{equation*}
with $a_{p,d}=\min\left\{\frac{1}{2},\frac{1}{d-1}\left(\frac{d}{4}+\frac{p}{2}-1\right)\right\}$.

Combining this Hessian estimate with the Bochner-type identity \eqref{eq.Bochner} and simplifying, we can obtain that, at the maximum point $x_*$, 
\begin{align*}
    0\geq &f^{\frac{p}{2}-1}\eta^{p-1}\L(Q)
    \\\geq&\frac{2(1-\epsilon)}{d-1}Q^p-\frac{40}{R^2}\left(\frac{(p-d)^2}{2\epsilon(d-1)}+c_{p,d}\right)Q^{p-1}
    \\&-\frac{80(d+p-2)+40\max\{p-1,1\}}{R^2}Q^{p-1}
\end{align*}
for any $\epsilon\in(0,1)$, where
\begin{equation*}
    c_{p,d}=\left(\frac{2(p/2-1)^2}{d-1}+\frac{p(p-2)}{2}\right)_-+\left(\frac{p}{2}+1\right)\max\{p-1,1\}
\end{equation*}
remains bounded as $p\downarrow 1$.

Taking $\epsilon=1/2$, we obtain
\begin{equation}
    \sup_{B_{\frac{R}{2}}}|\nabla v|^2\leq Q(x_*)\leq\frac{c_pd^2}{R^2}
\end{equation}
for some constant $c_p$ which remains bounded as $p\downarrow 1$. This proves Proposition \ref{prop.loggradient}. As a corollary, we have the Harnack inequality with $H_2(p,d)\leq\exp\left(\frac{Cd}{p-1}\right)$ for $p\in(1,p_0)$.

\begin{remark}
    To see that this dependence on $d$ and $p$ is optimal in this regime, consider the fundamental solution $u(x)=|x|^{-\lambda}$, where $\lambda=\frac{d-p}{p-1}$. The function $u$ is positive and $p$-harmonic on $\Rd\setminus\{0\}$. For $p\in (1,d)$, we have 
    \begin{equation*}
        \frac{\sup_{B_1(2e_1)} u}{u(2e_1)}=2^\lambda,
    \end{equation*}
    which shows that any admissible Harnack constant must satisfy $H_2(p,d)\geq \exp\left(c\frac{d-p}{p-1}\right)$. In particular, when $p<(1-\delta)d$, this gives $H_2(p,d) \geq\exp\left(c'\frac{d}{p-1}\right)$, which shows the optimality of the log-gradient bound.
\end{remark}

\appendix
% This appendix uses the notation and theorem environments of main.tex.
% Its proof is independent of the convergence and stopping-time lemmas.
\section{Measurability of the game value}\label{sec.measurability}

We justify the measurability and strategy selections used in
Lemma~\ref{lem.kappastrategy}. Throughout this section the game, its boundary
payoff $F$, and the values $u^1_\epsilon,u^2_\epsilon$ are those of
Section~\ref{sec.tugofwar}; strategies are pure Borel strategies on the full
game history. The argument uses the fixed tug length and the spherical
noise specified there. Translation and rescaling give the same conclusions
for a game in any ball.

\begin{proposition}\label{prop.borelvalue}
Let $F$ be continuous and $0<\epsilon<(1+\rho)^{-1}$. Then
$u^1_\epsilon=u^2_\epsilon=:v_\epsilon$ is Borel measurable. Extend
$v_\epsilon$ to $\partial B_1(0)$ by $F$. With $A_\epsilon$ and $Q_\epsilon$
as in the proof of Lemma~\ref{lem.kappastrategy}, $Q_\epsilon$ is jointly
Borel on the graph of $A_\epsilon$, and
\begin{equation}\label{eq.borel.dpp}
 v_\epsilon(x)=\frac12\left(
 \sup_{a\in A_\epsilon(x)}Q_\epsilon(x,a)
 +\inf_{a\in A_\epsilon(x)}Q_\epsilon(x,a)\right).
\end{equation}
For every $\delta>0$ there are Borel maps
$a^\pm_\delta(x)\in A_\epsilon(x)$ such that, at every $x\in B_1(0)$,
\begin{align}\label{eq.borel.selectors}
 Q_\epsilon(x,a^+_\delta(x))
 &\geq\sup_{a\in A_\epsilon(x)}Q_\epsilon(x,a)-\delta,\notag\\
 Q_\epsilon(x,a^-_\delta(x))
 &\leq\inf_{a\in A_\epsilon(x)}Q_\epsilon(x,a)+\delta.
\end{align}
\end{proposition}

\begin{proof}
Recall that $d\geq2$ and $\rho=\sqrt{(d-1)/(p-1)}>0$.
Write $s=1-(1+\rho)\epsilon$, $C=\{x\in B_1(0):|x|>s\}$, and
$I=B_1(0)\setminus C$. On $C$ the game ends at the next move and has value
\[
 g(x)=\frac12\left(\sup_{z\in D_x}F(z)+\inf_{z\in D_x}F(z)\right),
 \qquad D_x=\partial B_1(0)\cap B(x,(1+\rho)\epsilon).
\]
The function $g$ is continuous on $C$: the extrema over $D_x$ equal those
over its closure, and these closed spherical caps vary continuously in
Hausdorff distance. Fixing a countable dense subset of $\partial B_1(0)$
and taking the first legal $\delta$-optimal point gives Borel approximate
maximizers and minimizers on $C$.

For $x\in I$ and $|a|=\epsilon$, set
\[
 Kf(x,a)=\mathbb E[f(x+a+W_a)],\qquad
 Tf(x)=\frac12\left(\sup_{|a|=\epsilon}Kf(x,a)
                       +\inf_{|a|=\epsilon}Kf(x,a)\right),
\]
where $W_a$ has the prescribed noise law, and set $Tf=g$ on $C$.
Choose $M>\|F\|_\infty$, put $L_0=-M$, $U_0=M$ on $I$ and
$L_0=U_0=g$ on $C$, and define $L_n=T^nL_0$, $U_n=T^nU_0$.
The operator $T$ is monotone and nonexpansive in the supremum norm, so
\begin{equation}\label{eq.borel.envelopes}
 -M\leq L_n\leq L_{n+1}\leq U_{n+1}\leq U_n\leq M.
\end{equation}

\emph{Borel finite-horizon values and choices.}
We give a common induction for all $d\geq2$. Let
\[
 \omega(t)=\sup\{|F(\xi)-F(\eta)|:
       \xi,\eta\in\partial B_1(0),\ |\xi-\eta|\leq t\}.
\]
Thus $\omega(t)\to0$ as $t\downarrow0$. Consider bounded Borel functions
$f$ which are continuous away from finitely many centered spheres, with
exceptional radii in a finite set $E$ containing zero.
This means that every possible discontinuity lies in
$\bigcup_{e\in E}\{x\in B_1(0):|x|=e\}$, a finite union of spheres
centered at the origin, with radius zero representing the single point $0$.
We also require the uniform angular estimate
\begin{equation}\label{eq.borel.angular}
 |f(Rx)-f(x)|\leq\omega(\|R-\mathrm{Id}\|),
 \qquad R\in SO(d),\quad x\in B_1(0),
\end{equation}
where the norm is the operator norm. This estimate implies continuity
on each individual sphere, including the exceptional ones: nearby points
on a sphere of positive radius are related by rotations tending to the
identity. Consequently these continuity properties already imply that
$f$ is Borel, by partitioning $B_1(0)$ into the exceptional spheres and
their complement.

The functions $L_0,U_0$ have these properties with $E=\{0,s\}$.
Indeed, $D_{Rx}=RD_x$ gives \eqref{eq.borel.angular} for $g$, and the
interior values are constant. Moreover, the angular estimate is preserved
by $T$: in the interior, rotate both the action and its noise, and use
\eqref{eq.borel.angular} before taking the two extrema; on $C$, use the
same estimate for $g$.

For any bounded Borel $f$, the map $(x,a)\mapsto Kf(x,a)$ is jointly
Borel. To see this, let $Z$ be a standard Gaussian vector and put
$\Pi_a=\mathrm{Id}-aa^\mathsf{T}/\epsilon^2$. The realization
\[
 W(a,Z)=\rho\epsilon\frac{\Pi_aZ}{|\Pi_aZ|},
\]
with value zero when the denominator vanishes, is jointly Borel and has
the required law; the exceptional event has probability zero for every
action. Integration against the fixed Gaussian law proves the assertion.
We next verify the remaining continuity property of $Tf$.

If $d\geq3$, $x\ne0$, and $a$ is not parallel to $x$, then
\[
 |x+a+W_a|^2=|x+a|^2+\rho^2\epsilon^2+2x\cdot W_a
\]
has an atomless distribution. Indeed, the projection of $x$ onto
$a^\perp$ is nonzero, and every level set of this nonconstant linear
functional on the noise sphere has surface measure zero. The successor
therefore avoids the exceptional spheres almost surely. In the Gaussian
realization, successors converge almost surely when $(x,a)$ varies
continuously, so dominated convergence proves joint continuity of $Kf$
at such a pair. For the radial actions $a=\pm\epsilon x/|x|$, the same
argument applies whenever
\[
 r_\pm(|x|):=\sqrt{(|x|\pm\epsilon)^2+\rho^2\epsilon^2}\notin E.
\]
Thus, outside the finite set of state radii
\[
 E'=\{0,s\}\cup r_+^{-1}(E)\cup r_-^{-1}(E),
\]
$Kf$ is jointly continuous at every $(x,a)$ with $x\in I$. Compactness of
the action sphere proves continuity of its supremum and infimum, and
hence of $Tf$; on $C$, continuity follows from that of $g$.

For $d=2$, let $J$ be rotation through $\pi/2$ and write
\[
 Y_\sigma(x,a)=x+a+\sigma\rho Ja\quad(\sigma=\pm1),\qquad
 b=\epsilon\sqrt{1+\rho^2}.
\]
Then $Kf(x,a)=\tfrac12\sum_{\sigma=\pm1}f(Y_\sigma(x,a))$.
For each $x\ne0$, the condition $|Y_\sigma(x,a)|=e$, with $e\in E$,
is equivalent to the line--circle intersection
\begin{equation}\label{eq.borel.planar.actions}
 (x-\sigma\rho Jx)\cdot a
     =\frac{e^2-|x|^2-b^2}{2},\qquad |a|=\epsilon.
\end{equation}
There are at most $4|E|$ such actions. They have explicit Borel
parametrizations: setting $h_\sigma=x-\sigma\rho Jx$ and
$c=(e^2-|x|^2-b^2)/(2\epsilon|h_\sigma|)$, the two solutions, when
$|c|\leq1$, are
\begin{equation}\label{eq.borel.planar.branches}
 a_{e,\sigma}^{\pm}(x)=\epsilon\left(
 c\frac{h_\sigma}{|h_\sigma|}
 \pm\sqrt{1-c^2}\frac{Jh_\sigma}{|h_\sigma|}\right).
\end{equation}
Assign any fixed legal action when $x=0$ or $|c|>1$.

Exclude $0,s$, the tangency radii $|e-b|,e+b$ for $e\in E$, and
the state radii where both successors are exceptional. The last set
is finite as well. In fact, if $|Y_+|=e$ and $|Y_-|=h$ with $e,h\in E$,
the triangle with vertices $0,Y_+,Y_-$ has fixed side lengths
$e,h,2\rho\epsilon$, so has at most two configurations up to rotations.
Since
\[
 x=\frac{Y_++Y_-}{2}+\frac{J(Y_+-Y_-)}{2\rho},
\]
each pair $(e,h)$ gives at most two possible values of $|x|$.
Denote the resulting finite set of excluded radii by $E'$.

We prove continuity of the extrema at any $x\in I$ with
$|x|\notin E'$. Suppose $x_k\in I$, $x_k\to x$, and $a_k\to a$.
If neither $Y_\sigma(x,a)$ is exceptional, ordinary continuity gives
$Kf(x_k,a_k)\to Kf(x,a)$. Otherwise exactly one successor is
exceptional, of radius $e>0$, and the corresponding intersection in
\eqref{eq.borel.planar.actions} is transverse. Put
$t_k=|Y_\sigma(x_k,a_k)|$. Formula~\eqref{eq.borel.planar.branches},
with $x$ fixed and $e$ replaced by $t_k$, gives actions $\widetilde a_k\to a$
satisfying $|Y_\sigma(x,\widetilde a_k)|=t_k$. The two successors of
radius $t_k$ approach the same nonzero point and hence differ by a
rotation tending to the identity. Estimate~\eqref{eq.borel.angular}
controls their $f$-values, while continuity controls the other successor.
Consequently, in both cases we can choose $\widetilde a_k\to a$ so that
\begin{equation}\label{eq.borel.matching}
 Kf(x_k,a_k)-Kf(x,\widetilde a_k)\longrightarrow0.
\end{equation}
Choosing almost maximizing $a_k$ and passing to a convergent subsequence
therefore proves upper semicontinuity of $\sup_a Kf(x,a)$.
Conversely, for any fixed action $a$ at $x$, the same formula, now
varying $x_k$ and keeping the exceptional radius $e$ fixed, gives
$a_k\to a$ with $Kf(x_k,a_k)\to Kf(x,a)$; if there is no exceptional
successor, take $a_k=a$. This proves lower semicontinuity of the
supremum. Applying the argument to $-f$ proves continuity of the infimum.
Thus $Tf$ is continuous away from the spheres with radii in $E'$,
also in dimension two.

Finally, the Borel choices are obtained in the same way in both
dimensions. Adjoin to a fixed countable dense set of actions the two
radial actions when $d\geq3$, or the finitely many maps
\eqref{eq.borel.planar.branches} when $d=2$, using a fixed action at
$x=0$. Enumerate these Borel maps as $a_j(x)$.
For $x\ne0$, $Kf(x,\cdot)$ is continuous except possibly at the listed
actions. At $x=0$ it is continuous by \eqref{eq.borel.angular} and
rotational invariance of the noise. Hence
\[
 \sup_{|a|=\epsilon}Kf(x,a)=\sup_j Kf(x,a_j(x)),\qquad
 \inf_{|a|=\epsilon}Kf(x,a)=\inf_j Kf(x,a_j(x)).
\]
Both extrema are Borel. Taking the first index whose value is within
$\delta>0$ of the corresponding extremum gives Borel approximate
maximizers and minimizers. Together with the boundary choices already
constructed on $C$, this proves by induction that every $L_n,U_n$ is
Borel and admits the required Borel choices.

Backward induction, using these selectors with summable errors, now
identifies $L_n,U_n$ as the finite-horizon values with cutoff payoffs
$-M,+M$, respectively. Reaching $C$ means playing the final boundary move;
otherwise the cutoff is paid after $n$ interior moves. In particular,
since the original payoff, including its value zero on nontermination,
lies between $-M$ and $M$,
\begin{equation}\label{eq.borel.sandwich}
 L_n\leq u^1_\epsilon\leq u^2_\epsilon\leq U_n.
\end{equation}
The finite-horizon strategies can be extended by any fixed Borel legal
rule, so this comparison uses only admissible strategies.

\emph{Uniform convergence of the envelopes.}
Put $D_n=U_n-L_n$ and $d_n=\|D_n\|_\infty$. By
\eqref{eq.borel.envelopes}, $d_n\downarrow d_*\geq0$. Suppose $d_*>0$.
For $\alpha>0$ define
\[
 m_{\alpha,n}=\sup\{U_n(x):D_n(x)\geq d_*-\alpha\},\qquad
 m_\alpha=\lim_{n\to\infty}m_{\alpha,n},\qquad
 m=\lim_{\alpha\downarrow0}m_\alpha.
\]
These limits exist by monotonicity, and the sets in the supremum are
nonempty. A diagonal choice gives $n_j\to\infty$ and $x_j$ such that
\begin{equation}\label{eq.borel.extreme}
 D_{n_j}(x_j)\longrightarrow d_*,\qquad
 U_{n_j}(x_j)\longrightarrow m.
\end{equation}
We record that for any random states $Y_j$ and integers $k_j\to\infty$,
\begin{equation}\label{eq.borel.ceiling}
 D_{k_j}(Y_j)\longrightarrow d_*\text{ in probability}
 \quad\Longrightarrow\quad
 \mathbb P(U_{k_j}(Y_j)>m+\eta)\longrightarrow0\quad(\eta>0).
\end{equation}
Indeed, on $\{D_{k_j}\geq d_*-\alpha\}$ we have
$U_{k_j}\leq m_{\alpha,k_j}$; first let $j\to\infty$ and then
$\alpha\downarrow0$. Boundedness also yields
$\limsup_j\mathbb E U_{k_j}(Y_j)\leq m$.

Consider any sequence satisfying \eqref{eq.borel.extreme}. Let
$A_j^+,A_j^-$ be the supremum and infimum of $K U_{n_j-1}(x_j,\cdot)$,
and define $B_j^+,B_j^-$ similarly using $L_{n_j-1}$.
Each difference $A_j^\pm-B_j^\pm$ belongs to $[0,d_{n_j-1}]$, and
their average tends to $d_*$. Hence both differences tend to $d_*$.
Choose an action within $1/j$ of $A_j^+$ and call its successor $Y_j$.
Then
\[
 \mathbb E D_{n_j-1}(Y_j)
 \geq A_j^+-B_j^+-1/j\longrightarrow d_*.
\]
Since $0\leq D_{n_j-1}\leq d_{n_j-1}\to d_*$, these gaps converge to
$d_*$ in probability. By \eqref{eq.borel.ceiling},
$\limsup_j A_j^+\leq m$. Also $A_j^+\geq U_{n_j}(x_j)\to m$, whence
\[
 A_j^+,A_j^-\longrightarrow m,\qquad
 B_j^+,B_j^-\longrightarrow m-d_*.
\]
Thus, under \emph{any} sequence of actions at $x_j$, the successor satisfies
\[
 \mathbb E U_{n_j-1}(Y_j)\longrightarrow m,\qquad
 \mathbb E D_{n_j-1}(Y_j)\longrightarrow d_*.
\]
The latter gives convergence of
the gaps in probability; the former, together with
\eqref{eq.borel.ceiling} and boundedness, gives
\begin{equation}\label{eq.borel.propagation}
 \bigl(D_{n_j-1}(Y_j),U_{n_j-1}(Y_j)\bigr)
 \longrightarrow(d_*,m)\quad\text{in probability}.
\end{equation}

The implication from \eqref{eq.borel.extreme} to
\eqref{eq.borel.propagation} is uniform near $(d_*,m)$: for every
$\eta,\zeta>0$ there are $0<\theta<d_*/2$ and $N$ such that, whenever $n\geq N$
and $|D_n(x)-d_*|+|U_n(x)-m|<\theta$, every action has probability at
least $1-\zeta$ of producing
\[
 |D_{n-1}(Y)-d_*|+|U_{n-1}(Y)-m|<\eta.
\]
Otherwise a sequence violating this assertion would contradict
\eqref{eq.borel.propagation}. Such source states lie in $I$, since $D_n=0$
on $C$. Conditioning and induction therefore propagate convergence to
$(d_*,m)$ through any fixed number of moves, also for a process frozen on
reaching $C$.
Now let both players always choose $a=\epsilon e_1$ until reaching $C$,
and freeze the process there. Every interior move increases the first
coordinate by exactly $\epsilon$, because the noise is perpendicular
to $e_1$. Thus $C$ is reached in at most
$K=\lceil2/\epsilon\rceil+1$ moves. Starting from the states in
\eqref{eq.borel.extreme}, propagation gives
$D_{n_j-K}(X_K)\to d_*$ in probability. But $X_K\in C$ and $D_n=0$
on $C$, a contradiction. Therefore $d_*=0$.

\emph{Identification and measurable strategies.}
By \eqref{eq.borel.sandwich}, the two game values coincide and are the
uniform limit $v_\epsilon$ of the Borel functions $L_n,U_n$. Thus
$v_\epsilon$ is Borel. Nonexpansiveness gives $Tv_\epsilon=v_\epsilon$,
which is \eqref{eq.borel.dpp}. The Gaussian representation above shows
joint Borel measurability of $Q_\epsilon$ on $I$; on $C$ it is simply
$Q_\epsilon(x,z)=F(z)$.
Given $\delta>0$, choose $n$ with
$\|U_n-v_\epsilon\|_\infty<\delta/4$ and a Borel
$\delta/2$-optimal maximizing selector for $K U_n$. The error in each
expectation is less than $\delta/4$, so this selector satisfies the first
inequality in \eqref{eq.borel.selectors} on $I$. Use the analogous
minimizing selector and the previously constructed boundary choices on
$C$. This completes the proof.
\end{proof}

In Lemma~\ref{lem.kappastrategy}, taking
$a_n=a^+_{\kappa2^{-n}}(X_{n-1})$ now defines a pure Borel strategy.
Equation~\eqref{eq.borel.dpp} supplies the dynamic programming inequality
used there, and the bounded submartingale and optional stopping argument
applies as written. Thus no measurable pasting of pointwise chosen
continuation strategies is needed.

\section*{Acknowledgements}
We are grateful to J. C. Mourrat and D. Pechersky for useful discussions.

\section*{Funding}
This work was supported by the National Natural Science Foundation of China under RFIS Grant No. W2531011. H. Wang was also supported by the National Natural Science Foundation of China under Grant Nos. 12595284 and 12595280.

\section*{Disclosure Statement}
The authors declare that they have no competing interests.

\bibliographystyle{abbrv}
\bibliography{Hopfref}

\end{document}